\documentclass[12pt]{amsart}

\author[E. Gnang]{Edinah K. Gnang}
\address{School of Mathematics, Institute for Advanced Study, 1 Einstein Drive, Princeton, NJ, USA}
\email{gnang@ias.edu}
\author[M. Radziwi\l\l]{Maksym Radziwi\l\l}
\address{School of Mathematics, Institute for Advanced Study, 1 Einstein Drive, Princeton, NJ, USA}
\email{maksym@ias.edu}
\author[C. Sanna]{Carlo Sanna}
\address{Universit\`a degli Studi di Torino, Department of Mathematics, Turin, Italy}
\email{carlo.sanna.dev@gmail.com}

\thanks{The first and second authors were partially supported by NSF grant DMS-1128155}

\title{Counting arithmetic formulas}

\usepackage{amsmath}
\usepackage{amssymb}
\usepackage{amsthm}
\usepackage{geometry}
\usepackage[colorlinks=true]{hyperref}
\usepackage{enumerate}
\usepackage{graphicx}
\usepackage{epstopdf}
\usepackage[section]{placeins}

\newtheorem{thm}{Theorem}[section]
\newtheorem{lem}[thm]{Lemma}
\newtheorem{cor}[thm]{Corollary}
\theoremstyle{definition}
\newtheorem{defi}{Definition}[section]

\def\Z{\mathbf{Z}}
\def\N{\mathbf{N}}
\def\R{\mathbf{R}}
\def\P{\mathbf{P}}
\def\l{\mathbf{l}}
\def\r{\mathbf{r}}
\def\T{\mathfrak{T}}

\def\lb{\operatorname{lb}}

\begin{document}

\begin{abstract}
An \emph{arithmetic formula} is an expression involving only the constant $1$, and the binary operations of addition and multiplication, with multiplication by $1$ not allowed. 
We obtain an asymptotic formula for the number of arithmetic formulas evaluating to $n$ as $n$ goes to infinity, solving a conjecture of E.~K.~Gnang and D.~Zeilberger \cite{EdiZ}.
We give also an asymptotic formula for the number of arithmetic formulas evaluating to $n$ and using exactly $k$ multiplications.
Finally we analyze three specific encodings for producing arithmetic formulas. For almost all integers $n$, we compare the lengths of the arithmetic formulas for $n$ that each encoding produces with the length of the shortest formula for $n$ (which we estimate from below). 
We briefly discuss the time-space tradeoff offered by each. 
\end{abstract}

\maketitle

\section{Introduction}
\subsection{Counting arithmetic formulas}
An \textit{arithmetic formula} is an expression involving only the constant $1$
and the binary operations of addition and multiplication, with multiplication by $1$ not allowed. 
For example, $4$ has exactly $6$ arithmetic formulas, 
\begin{gather*}
\begin{array}{ccc}
1 + (1 + (1 + 1)), & 1 + ((1 + 1) + 1), & (1 + (1 + 1)) + 1,
\end{array} \\
\begin{array}{cccc}
((1 + 1) + 1) + 1, & (1 + 1) + (1 + 1), & (1 + 1) \times (1 + 1) 
\end{array}
\end{gather*}
A systematic study of arithmetic formulas was initiated by Patrick, Gnang and Zeilberger \cite{EdiZ}~\cite{EdiK}.
The number of arithmetic formulas evaluating to $n$ using only addition corresponds to the number of ways one can place a sequence of parentheses in the sum $1 + 1 + \cdots + 1$, containing $n$ times the number $1$. It is well known that there is $C_{n-1}$ ways of doing this, where $C_n$ is the Catalan number
\cite[Ch.~6, Corollary~6.2.3]{Sta01},
\begin{equation*}
C_m := \frac{1}{m + 1} \binom{2m}{m} \sim \frac{1}{\sqrt{\pi}}
\frac{4^m}{m^{3/2}} \text{ as } m \rightarrow \infty.
\end{equation*}
On the other hand, the number of arithmetic formulas for $n$ using addition and multiplication is more mysterious. 
It was conjectured by Gnang and Zeilberger \cite{EdiZ} that there is an asymptotic of the form $c \cdot \rho^n \cdot n^{-3/2}$, with two constants $c > 0$ and $\rho > 4$ (with $\rho$ most likely a transcendental number). Our first result is a proof of this conjecture. 
Let $f(n)$ be the number of arithmetic formulas for $n$ \cite{Hei13}. 
\begin{thm}\label{thm:f_asymp}
There exists constants $c > 0$ and $\rho > 4$ 
such that
\begin{equation*}
f(n) \sim \frac{c \rho^n}{n^{3/2}} ,
\end{equation*}
as $n \rightarrow \infty$. In fact,
\begin{align*}
c & = 0.145691854699979 \ldots \\
\rho & = 4.076561785276046 \ldots
\end{align*}
\end{thm}
In addition our method gives an asymptotic expansion for $f(n)$. We refer
the reader to the proof of Theorem~\ref{thm:f_asymp} for more details. 
Theorem~\ref{thm:f_asymp} is also motivated by some relations with the factoring problem, see
Section~\ref{sec:factor}. 

We obtain a completely explicit characterization of the constant $\rho$. 
It is determined as $\rho := 1/\xi$ where $0 < \xi < 1/4$ is the smallest positive solution to the equation $\widetilde{F}(\xi) = 1 / 4$, with
\begin{equation*}
\widetilde{F}(z) := z + \sum_{d = 2}^\infty f(d) (F(z^d) - z^d) \text{ and } F(z) := \sum_{n = 1}^\infty f(n) z^n .
\end{equation*}
The proof of Theorem~\ref{thm:f_asymp} can be easily adapted to count the number of arithmetic formulas in which also exponentiation is allowed (and such that $1$ is never an argument of exponentiation). We call such formulas \textit{arithmetic exponential formulas}. An analogue of Theorem~\ref{thm:f_asymp} holds for counting arithmetic exponential formulas but with a larger $\rho = 4.13073529514801\ldots$. 

The proof of Theorem~\ref{thm:f_asymp} depends on generating functions and complex analysis.
A natural idea is to produce an elementary proof of Theorem~\ref{thm:f_asymp}  by first
asking for the number $f_k(n)$ of arithmetic formulas for $n$ using only addition and exactly $k$ multiplication operations. This we achieve in the theorem below.

\begin{thm}\label{thm:fk_rude_asymp}
For all integers $k \geq 0$, we have
\begin{equation*}
f_k(n) \sim \frac{\sigma^k}{4 \sqrt{\pi} \, k!}\, 4^n n^{k-3/2},
\end{equation*}
as $n \to +\infty$, where
\begin{equation*}
\sigma := \sum_{m=1}^\infty \frac1{4^{m-1}} \! \sum_{\substack{d \, \mid \, m \\ 1 < d < m}} f_0(d) f_0(m/d) .
\end{equation*}
\end{thm}
One would like to sum the above formula over all $k$, assuming
sufficient uniformity, and claim that $\rho$ in Theorem~\ref{thm:f_asymp} is equal to  $4 e^{\sigma}$.
However $\rho < 4 e^{\sigma}$ and therefore for large $k$ there occurs a significant break in the
uniformity of Theorem \ref{thm:fk_rude_asymp}. This is 
expected since, for example,  $f_k(n) = 0$ for $k > \log n/\log 2$. 

\subsection{Factoring}\label{sec:factor}
One motivation for our work comes from factoring. 
For a given positive integer $n$ one would like to understand the following graph $G_n$: 
The nodes of the graph $G_n$ correspond to the various arithmetic formulas for $n$ and an edge is placed between two nodes if one can pass from one 
formula to the other by using only one operation of either associativity, distributivity or commutativity. 

One can depict arithmetic formulas as full binary trees, so that the graph $G_n$ is a graph whose vertices correspond to certain special full binary trees.
Various arithmetic algorithms such as integer factoring algorithms can be depicted as walks starting from some particular vertex of the graph $G_n$ (say the one corresponding to the recursive Horner encoding, see below for a definition of this encoding) and terminating
at a vertex associated with a formula encoding of $n$ whose corresponding tree is rooted at a multiplication node. 


A vertex $v$ of $G_n$ corresponding to an arithmetic formula using only additions has the largest possible degree in $G_n$, precisely $\deg(v) = f_0(n) - 1$.
So in order to understand the connectivity of the graph $G_n$ we compare
$f_0(n) - 1$ to the order of the graph $G_n$. The order of the graph $G_n$
corresponds to the number $f(n)$ of representations of $n$ using only $1$'s
and operation of addition and multiplication. Therefore as an immediate
consequence of Theorem~\ref{thm:f_asymp} we obtain the following
\begin{cor}
Let $C = \rho / 4 = 1.019140446319 \ldots$. 
Then, for some constant $c > 0$, as $n \rightarrow \infty$, 
\begin{equation*}
\max_{v \in G_n} \deg(v) \sim c \cdot \frac{|G_n|}{C^n}.
\end{equation*}
\end{cor}

Of particular interest in the graph $G_n$ are formulas which are short
because they minimize the space needed for encoding $n$.  

\subsection{Shortest encodings} 
We will discuss three special monotone formula encoding schemes
called the \emph{first canonical form} or Goodstein encoding \cite{Goodstein},
the \emph{second canonical form} \cite{EdiZ} and the \emph{Horner encoding}.
We will focus on \textit{arithmetic exponential formulas} (that is, arithmetic formulas allowing exponentiation), because a lower bound for the lengths of such formulas is also a lower bound for the length of the shortest arithmetic formula with only addition and multiplication allowed.

The Goodstein encoding consists in writing the binary expansion of
an integer $n=\sum_{i}2^{a_{i}}$ and recursively writing down the
binary expansion for each integer $a_{i}$ until we obtain a representation
of $n$ as formula involving only $2$ and $1$'s, the final step
will consist in replacing each $2$ by $1+1$ thereby obtaining a
monotone formula encoding of $n$ which only uses additions ($+$)
and exponentiations ($\wedge$) gates and has input 1. For
example the Goodstein encoding for the number $31$ corresponds to
\begin{equation*}
31=\left(1+1\right)^{\left(\left(1+1\right)^{\left(1+1\right)}\right)}+\left(\left(1+1\right)^{\left(\left(1+1\right)+1\right)}+\left(\left(1+1\right)^{\left(1+1\right)}+\left(\left(1+1\right)+1\right)\right)\right) .
\end{equation*}
By contrast to the Goodstein encoding, the second canonical form of
an integer $n$ is slightly more intricate. We start by writing down
the prime factorization $n=p_{1}^{\alpha_{1}}\cdots p_{r}^{\alpha_{r}}$
and subsequently we express each prime as $1+\left(p_{i}-1\right)$.
Finally we recursively apply this scheme to every $\left(p_{i}-1\right)$
and every exponent $\alpha_{i}$. Thus we obtain a monotone formula
encoding for $n$ which uses a combination of addition ($+$), multiplication
($\times$), and exponentiation gates ($\wedge$) and input
restricted to $1$. As an example we express the second canonical form
associated to $2430$ 
\begin{equation*}
2430=\left(\left(1+1\right)\times\left(1+\left(1+1\right)\right)^{\left(1+\left(1+1\right)^{\left(1+1\right)}\right)}\right)\times\left(1+\left(1+1\right)^{\left(1+1\right)}\right) .
\end{equation*}
In \cite{EdiZ} it was observed that for most integers $n$ the second canonical
form is smaller than the Goodstein encoding. Our next result provides
some theoretical validation for this empirical observations. Let $S_{\mathrm{short}}\left(n\right)$
denote the length of the shortest monotone formula encoding of $n$,
let $S_{\mathrm{SCF}}\left(n\right)$ and $S_{\mathrm{FCF}}\left(n\right)$
denote respectively the size of the first and second canonical form
encoding of $n$. The special interest in formula sizes stems from
the connection between circuit complexity and integers encoding schemes.
Building on a sequence of constructions by Cheng \cite{Cheng} and Koiran \cite{Koiran},
Burgisser \cite{Burgisser} showed that if the sequence of integers $\left\{ n!\right\} _{n \in \N}$
is hard to compute, then any algebraic circuits for computing the
permanent of a sequence $\left\{ M_n \in \Z^{n\times n}\right\} _{n \in \N}$
of matrices using addition ($+$) and multiplication ($\times$) gates
with input restricted to $\left\{ -1,\,1\right\} $, must have superpolynomial
size.\\ Also related results (for example, for circuits) have
been obtained in \cite{Moreira} and \cite{Svaiter}. We
refer the reader to the references there-in for further information on this
topic. 

\begin{thm}\label{thm:S_short_bound}
For almost all positive integers $n$ we have
\begin{equation*}
S_{\mathrm{short}}\left(n\right) \geq \frac{\log n}{\log 4} .
\end{equation*}
Precisely, given $\varepsilon > 0$, the number of integers $n \leq x$ such that
\begin{equation*}
S_{\mathrm{short}}\left(n\right) \leq (1 - \varepsilon) \frac{\log n}{\log 4}
\end{equation*}
is $O(x^{1-\varepsilon})$, as $x \to +\infty$.
\end{thm}

\begin{thm}\label{thm:S_FCF_asymp}
Given $\varepsilon > 0$, for almost all positive integers $n$,
\begin{equation*}
S_{\mathrm{FCF}}(n) \geq \left(\frac{1}{4(\log 2)^2} - \varepsilon\right) \cdot \log n \log\log n .
\end{equation*}
\end{thm}

\begin{thm}\label{thm:S_SCF_bound}
For all integers $n \geq 2$, we have
\begin{equation*}
S_{\mathrm{SCF}}\left(n\right)\le 6 \, \frac{\log n}{\log 2} .
\end{equation*}
\end{thm}
In conclusion, while the first canonical form is rapid it provides
formulas of sub-optimal length compared to the shortest formula. 
The second canonical form is more computationally intensive but gives
rise to shorter formulas, of quality comparable to the shortest formula.
The drawback is computational complexity, and this drawback is alleviated 
by
the Horner encoding, which is obtained from a recursive factoring
of the Goodstein encoding. 
We write below the recursive Horner encoding of the integer $53376$ 
\begin{equation*}
53376 ={\left({\left({\left((1+1) + 1\right)} (1+1)^{(1+1)} + 1\right)} (1+1)^{(1+1)^{(1+1)} + 1} + 1\right)} (1+1)^{{\left((1+1) + 1\right)} (1+1) + 1} 
\end{equation*}
The properties of the recursive Horner encoding are similar to the second canonical form. 
For example we obtain essentially the same results for $S_{\mathrm{Hor}}(n)$ as for
$S_{\mathrm{SCF}}(n)$. We suspect however that the second canonical form gives
on average slightly shorter formulas than the Horner encoding. We think it's
an interesting question but we didn't pursue it.
Finally we note that one can efficiently recover
recursive Horner encodings from Goodstein encodings.

\subsection*{Notation}
Hereafter, $\N$ denotes the set of positive integers and $\N_0 := \N \cup \{0\}$.
We use the Landau--Bachmann $o$ and $O$ symbols, as well as Vinogradov’s $\ll$ notation, with their usual meanings.
We adopt the usual convention that empty sums and empty products, e.g.\ $\sum_{n=x}^y$ and $\prod_{n=x}^y$ with $x > y$, have values $0$ and $1$, respectively.
Moreover, we employ the convention that a binomial coefficient $\binom{a}{b} = 0$ if $a < b$.
Finally, if $g$ and $h$ are two arithmetic functions, we write $g *^\prime h$ for their \emph{proper} Dirichlet convolution (cf.~\cite[Ch.~2]{Hua09}), i.e., the function defined by
\begin{equation*}
(g *^\prime h)(n) := \sum_{\substack{d \, \mid \, n \\ 1 < d < n}} g(d) \, h(n/d), \quad n \in \N ,
\end{equation*}
where the sum runs over all the proper divisors $d$ of $n$.

\section{Preliminaries}

First of all, we need a rigorous formal definition of what arithmetic formulas are.

\begin{defi}
Let $n$ be a positive integer. 
An \emph{arithmetic formula} $A$ for $n$ is an \mbox{$\N$-valued} $\{+,\times\}$-labeled full binary tree such that:
\begin{enumerate}[(i).]
\item The value of the root is $n$.
\item The value of each leaf is $1$.
\item All node except the leaf nodes are labelled with a $+$ (\emph{additive node}) or $\times$ (\emph{multiplicative node}).
\item The value of each additive node is $a + b$, where $a$ and $b$ are the values of its children.
\item The value of each multiplicative node is $ab$, where $a$ and $b$ are the values of its children.
\item If $a$ and $b$ are the values of the children of a multiplicative node, then $a,b \geq 2$.
\end{enumerate}
Similarly, an \emph{arithmetic exponential formula} $E$ for $n$ is an \mbox{$\N$-valued} $\{+,\times,\wedge\}$-labeled full binary tree that satisfies all the previous points, only with (iii) slightly modified to
\begin{enumerate}[(i).]
\item[(iii').] All nodes except the leaf nodes are labelled with $+$ (\emph{additive node}), $\times$ (\emph{multiplicative node}) or $\wedge$ (\emph{exponential node}).
\end{enumerate}
and furthermore
\begin{enumerate}[(i).]
\item[(vii).] The value of an exponential node is $a^b$, where $a$ and $b$ are the values of its left and right children, respectively.
\item[(viii).] If $a$ and $b$ are the values of the children of an exponential node, then $a,b \geq 2$.
\end{enumerate}
Finally, we say that a multiplicative node of $A$ or $E$ is \emph{primitive} if it has no multiplicative ancestor.
\end{defi}

Now we can also define the length of an arithmetic formula.

\begin{defi}
The \emph{size} or \emph{length} of an arithmetic formula (or an arithmetic exponential formula) $A$ is the number of nodes of $A$;
equivalently, the number of symbols $1$, $+$, $\times$ and $\wedge$ needed to write $A$ in the usual infix notation, or in Polish notation.
Note that parenthesis do not count.
\end{defi}

We state below a frequently used 
immediate consequence of Stirling's formula.

\begin{lem}\label{lem:f0_asymp}
We have,
\begin{equation*}
f_0(n) = C_{n-1} \sim \frac1{4\sqrt{\pi}}\frac{4^n}{n^{3/2}} ,
\end{equation*}
as $n \to +\infty$.
\end{lem}

\section{Proof of Theorem~\ref{thm:f_asymp}}

We start with a couple of lemmas.
For all integers $n \geq 2$, we denote by $f^+(n)$, respectively $f^\times(n)$, the number of arithmetic formulas for $n$ which root node is additive, respectively multiplicative.
We set also $f^+(1) := 1$ and $f^\times(1) := 0$.
Thus, obviously, $f(n) = f^+(n) + f^\times(n)$, for all positive integers $n$.
Moreover, it is easily seen that

\begin{lem}\label{lem:f_recursion}
For all integers $n \geq 2$, it results
\begin{equation*}
f^+(n) = \sum_{h=1}^{n-1} f(n - h) f(h)
\end{equation*}
and $f^\times(n) = (f *^\prime f)(n)$.
\end{lem}

The next lemma is a first upper bound on $f(n)$ which we need to be sure that the radius of convergence of $F(z)$ is positive.

\begin{lem}\label{lem:f_upper_bound}
We have $f(n) < 8^n$, for each positive integer $n$.
\end{lem}
\begin{proof}
Consider that an arithmetic formula for $n$, thought of as a full binary tree, has at most $n-1$ non-leaf nodes.
For any nonnegative integer $k$ there are exactly $C_k$ full binary trees with $k$ non-leaf nodes.
Given one of them, its non-leaf nodes can be labeled (as additive or multiplicative) in $2^k$ different ways.
In conclusion, since $C_k \leq 4^k$, we get
\begin{equation*}
f(n) \leq \sum_{k=0}^{n-1} 2^k C_k \leq \sum_{k=0}^{n-1} 8^k < 8^n .
\end{equation*}
\end{proof}
As for the analytic input into our proof
we will need the following version of ``Darboux's method''. 
\begin{lem}[Darboux's method]\label{lem:darboux}
Let $v(z)$ be analytic in some disk $|z| \leq 1 + \eta$, and suppose
that in a neighborhood of $z = 1$ it has the expansion $v(z) =
\sum_{j=0}^\infty v_j (1-z)^{j}$. Let $\beta \notin \{0,1,2,\ldots\}$. Then, the
$n$-th coefficient of $(1 - z)^{\beta} v(z)$ is equal to
$$
\sum_{j = 0}^{m} v_j \binom{n - \beta -j - 1}{n} + O(n^{-m-\beta-2}). 
$$
\end{lem}
\begin{proof}
See \cite[Theorem 5.3.1]{Wilf}.
\end{proof}
We will also need the following classical result of Pringsheim.
\begin{lem}\label{lem:pringsheim}
Let $f(z)$ be a power series with finite radius of convergence $R > 0$. If all
of the coefficients of $f(z)$ are nonnegative, then, $z = R$ is a
singular point. 
\end{lem}
\begin{proof}
See \cite[Chapter 8]{Remmert}.
\end{proof}
We will use the following immediate consequence of Lemma~\ref{lem:pringsheim} 
theorem: if
$f(z)$, a power series with nonnegative coefficients, has an analytic continuation to $|z| < R + \eta$, for some
$\eta > 0$, then the abscissa of the first singularity of $f(z)$
on the axis $x > 0$ is equal to the radius of convergence $R$. 
Now we are ready to prove Theorem~\ref{thm:f_asymp}.
\begin{proof}[Proof of Theorem~\ref{thm:f_asymp}]
Let $R$ be the radius of convergence of the generating function $F(z)$.
First of all $R \leq 1/4$ since $f(n) \geq f_0(n)$ and $f_0(n) > (4 - \varepsilon)^{n}$ for any $\varepsilon > 0$ and all $n$ large enough. On the other hand from Lemma~\ref{lem:f_upper_bound} we know that $R \geq 1/8$.
For each integer $d \geq 2$, it results that $F(z^d) - z^d$ has radius of convergence $R^{1/d} \geq R^{1/2}$.
Hence, for any $\delta > 0$ and $|z| < R^{1/2} - \delta$ we have $|F(z^d) - z^d| \ll_{\delta} |z|^{2d}$ and $f(d) < (1/R + \varepsilon)^d$ for sufficiently large $d$.
Therefore, the series $\widetilde{F}(z)$ converges absolutely for $|z| < R^{1/2}$ and it is analytic in that region, note also that $R^{1/2} > R$.
For $|z| < R$, from Lemma~\ref{lem:f_recursion} we obtain
\begin{align*}
\sum_{n=1}^\infty f^+(n) z^n = z + \sum_{n=2}^\infty \sum_{k=1}^{n-1} f(n - k) f(k) z^n = z + F(z)^2 ,
\end{align*}
while
\begin{align*}
\sum_{n=1}^\infty f^\times(n) z^n &= \sum_{n=1}^\infty \sum_{\substack{d \,\mid\, n \\ 1 < d < n}} f(d) f(n / d) z^n \\
&= \sum_{d=2}^\infty f(d) \sum_{m=2}^\infty f(m) z^{dm} = \sum_{d=2}^\infty f(d) (F(z^d) - z^d) .
\end{align*}
Thus,
\begin{equation*}
F(z) = \sum_{n=1}^\infty f(n) z^n = \sum_{n=1}^\infty f^+(n) z^n + \sum_{n=1}^\infty f^\times(n) z^n = F(z)^2 + \tilde{F}(z) ,
\end{equation*}
so that
\begin{equation}\label{equ:F_func_equ}
F(z)^2 - F(z) + \tilde{F}(z) = 0 .
\end{equation}
Taking into account that $F(0) = \widetilde{F}(0) = 0$, we can solve the quadratic equation (\ref{equ:F_func_equ}) and get
\begin{equation}\label{equ:F_sqrt}
F(z) = \frac{1 - \sqrt{1 - 4\widetilde{F}(z)}}{2} , \quad \mbox{ for } |z| < R.
\end{equation}
Since the coefficients of $F(z)$ are all positive, by Lemma~\ref{lem:pringsheim} we have that $F(z)$ has a singularity at $z = R$.
As observed before, in the region $|z| < R^{1/2}$ the function $\widetilde{F}(z)$ is analytic and $R^{1/2} > R$, thus providing an analytic continuation
of $F(z)$ to the larger region $|z| < R^{1/2}$. 
From (\ref{equ:F_sqrt}) we expect that the first singularity of $F(z)$ on the positive real axis occur at the point $\xi$ at which we have $\widetilde{F}(\xi) = 1/4$.
Such $\xi$ clearly exists because  $\widetilde{F}(x) > x$, for $x > 0$,
so that $\xi < 1/4$, while
 $\widetilde{F}(z)$ is analytic in $|z| < 1/\sqrt{8} \leq R^{1/2}$. 
We notice also that the root $\xi$ is simple, 
because $F(x)$ is increasing and analytic
on the segment $0 \leq x < 1/\sqrt{8}$. Thus we can write,
$$
1 - 4 \widetilde{F}(z) = (1 - z/\xi) G(z)
$$
for some $G(z)$, analytic in $|z| < 1/\sqrt{8}$ and non-vanishing on $0 \leq x < 1/\sqrt{8}$. 
As mentioned earlier, the formula
$$
F(z) = \frac{1 - \sqrt{(1 - z/\xi)G(z)}}{2}
$$
provides an analytic continuation of $F(z)$
to the larger disc $|z| < 1/\sqrt{8}$, since the radius
of convergence of $F(z)$ satisfies $R \leq 1/4 < 1/\sqrt{8}$. 
As an immediate application of Lemma~\ref{lem:pringsheim} the first singularity
of $F(z)$ on the positive real axis corresponds to the radius of convergence
$R$. Thus $R = \xi$. 
Before applying Lemma~\ref{lem:darboux} we need to say a few things about the location
of the zeros of $G(z)$. Since $\widetilde{F}(z)$ has positive and never vanishing
coefficients, we have
$
|\widetilde{F}(r e^{i\theta})| < \widetilde{F}(r) \leq \widetilde{F}(R)
$ 
for all $\theta \neq 0$ and $r \leq R$. Using this 
we notice that for $z \neq \xi$, and
$|z| \leq \xi$, 
$$
|1 - 4\widetilde{F}(z)| \geq 1 - 4 |\widetilde{F}(z)| > 1 - 4 \widetilde{F}(\xi)
= 0.
$$ 
It follows
that $G(z)$ has no zeros in $|z| \leq \xi$. By analyticity
this implies that there exists a neighborhood $|z| \leq \xi + \eta$
for some $\eta > 0$, where $G(z)$ doesn't vanish, in particular
$\sqrt{G(z)}$ is well-defined and doesn't vanish there.  Now, 
applying Lemma~\ref{lem:darboux} to $F(z\xi)$, or rather more precisely
applying Lemma~\ref{lem:darboux} to $-\sqrt{(1-z)G(z\xi)}/2$ (which has radius of convergence equal to $1$ and differs from $F(z)$ only at the constant term) we conclude
that for any $m > 0$, and $n \rightarrow \infty$, 
$$
f(n) \xi^n = -
\sum_{j = 0}^{m} c_j \binom{n - j - 3/2}{n} + O(n^{-m - 5/2})
$$
where the coefficients $c_j$ are obtained by writing 
$$
\sqrt{G(z\xi)} = \sum_{j \geq 0} c_j (1 - z)^{j}
$$
in a small neighborhood of $z = 1$. Since, as $n \rightarrow \infty$,
$$
\binom{n - j - 3/2}{n} \sim \frac{a_j}{n^{j + 3/2}}
$$
with $a_j \neq 0$, the claim follows. 
\end{proof}
\section{Proof of Theorem~\ref{thm:fk_rude_asymp}}

We will in fact prove a result stronger than Theorem~\ref{thm:fk_rude_asymp}. However before stating it, we introduce the concept of a $k$-\emph{trace}.

\begin{defi}\label{defi:k_trace}
Let $k$ be a positive integer.
A $k$-\emph{trace} is triple $(p,\l, \r)$ where $p$ is a positive integer and $\l, \r \in \N_0^p$ are tuples such that
$\ell_1 + r_1 + \ldots + \ell_p + r_p + p = k$. 
We denote by $\T_k$ the set of all $k$-\emph{traces}.
We define also $\T_0 := \{(0,0,0)\}$ so that $(0,0,0)$ can be thought of as the only $0$-trace.
\end{defi}

We are ready to state our asymptotic formula for $f_k(n)$.

\begin{thm}\label{thm:fk_asymp}
For all integers $k \geq 0$, we have
\begin{equation*}
f_k(n) \sim \frac1{4\sqrt{\pi}} \, \frac{4^n}{n^{3/2}} \sum_{(p,\l, \r) \in \T_k} \frac{n^p}{p!} \prod_{i=1}^p \left(\sum_{t=1}^\infty \frac{(f_{\ell_i} *^\prime f_{r_i})(t)}{4^{t-1}}\right) ,
\end{equation*}
as $n \to +\infty$.
\end{thm}

Observe that Theorem~\ref{thm:fk_rude_asymp} follows immediately from Theorem~\ref{thm:fk_asymp}, since for any $k \in \N$ the only $(p,\l, \r) \in \T_k$ with $p \geq k$ is $(k,\mathbf{0},\mathbf{0})$.
The next definition connects $k$-traces to arithmetic formulas.

\begin{defi}\label{defi:trace_of_A}
Suppose that $A$ is an arithmetic formulas for $n$ with $k$ multiplicative nodes.
If $k=0$ then the \emph{trace} of $A$ is $(0,0,0)$.
If $k \geq 1$, let $N_1, \ldots, N_p$ be the primitive nodes of $A$, ordered from left to right (there is no ambiguity since no primitive node is the ancestor of another primitive node).
For $i=1,\ldots,p$, let $\ell_i$, respectively $r_i$, be the number of multiplicative nodes in the left, respectively right, subtree of $N_i$.
Then the trace of $A$ is the triple $(p,\l,\r)$, with $\l = (\ell_1, \ldots, \ell_p)$ and $\r = (r_1, \ldots, r_p)$.
Finally, for all $k \in \N_0$ and $(p,\l,\r) \in \T_k$ we denote by $f_{(p,\l,\r)}(n)$ the number of arithmetic formulas for $n$ with trace $(p,\l,\r)$.
\end{defi}

It is easy to see that Definition \ref{defi:k_trace} and \ref{defi:trace_of_A} are consistent to each other, i.e., if $A$ is an arithmetic formula with $k$ multiplicative nodes then the trace of $A$ is actually a $k$-trace.

We give now a combinatorial formula for $f_{(p,\l,\r)}$ in terms of $f_0$ and $f_{\ell_i}$, $f_{r_i}$.

\begin{lem}\label{lem:fplr_formula}
For $k \in \N$ and $(p,\l,\r) \in \T_k$, we have
\begin{equation*}
f_{(p,\l,\r)}(n) = \sum_{n_1 + \cdots + n_p + m = n + p} \binom{m}{p} f_0(m) \prod_{i=1}^p (f_{\ell_i} *^\prime f_{r_i})(n_i) ,
\end{equation*}
where the sum runs over all $n_1, \ldots, n_p, m \in \N$ such that $n_1 + \cdots + n_p + m = n + p$.
\end{lem}
\begin{proof}
The general arithmetic formula $A$ evaluating to $n$ and with trace $(p,\l,\r)$ is depicted in Fig.~\ref{img:arith_rep_of_n}, where $n_1, \ldots, n_p$ are all the primitive multiplicative nodes of $A$ (we identify the nodes with their values since there is no risk of confusion).
Set $m := n - (n_1 + \cdots + n_p) + p$. 
On the one hand, if we remove from $A$ all the nodes below $n_1, \ldots, n_p$ we get a full binary tree with $m$ leaves.
There are exactly $f_0(m)$ such trees (addition is associative) and the nodes $n_1, \ldots, n_p$ can be attached to the leaves of each of them in $\binom{m}{p}$ different ways.
On the other hand, any subtree of $A$ with root $a_i$, respectively $b_i$, is an arithmetic formula for $a_i$, respectively $b_i$, and there are exactly $f_{\ell_i}(a_i)$, respectively $f_{r_i}(b_i)$, such arithmetic formulas.
Hence, since $a_i b_i = n_i$, there are $(f_{\ell_i} *^\prime f_{r_i})(n_i)$ possible subtrees of $n_i$.
All these choices are independent so the claim follows.
\end{proof}

\begin{figure}[h!]
\centering
\includegraphics[scale=1.0]{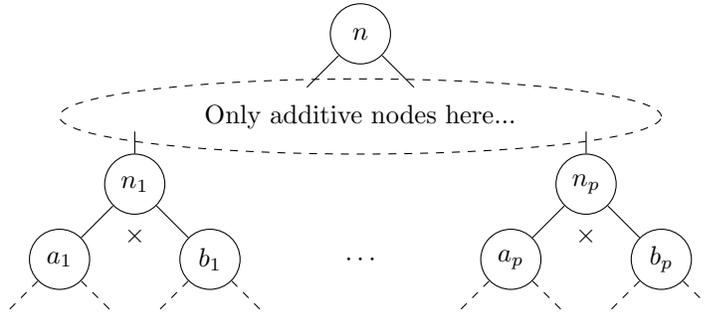}
\caption{An arithmetic formula for $n$.}
\label{img:arith_rep_of_n}
\end{figure}

\FloatBarrier

The next lemma is an easy upper bound on the proper Dirichlet convolution of two arithmetic functions.

\begin{lem}\label{lem:conv_bound}
Let $g$ and $h$ be arithmetic functions such that $g(n), h(n) \ll 4^n n^s$ for $n \in \N$, with $C > 0$ and $s \in \R$.
Then $(g *^\prime h)(n) \ll 3^n$ for $n \in \N$.
\end{lem}
\begin{proof}
We have
\begin{align*}
(g *^\prime h)(n) &= \sum_{\substack{d \,\mid\, n \\ 1 < d < n}} g(d) \, h(n/d) \ll n^s \sum_{\substack{d \,\mid\, n \\ 1 < d < n}} 4^{d+n/d} \ll 2^n n^s \sum_{\substack{d \,\mid\, n \\ 1 < d < n}} 1 \\
&\ll 2^n n^{s+1}  \ll 3^n ,
\end{align*}
since $d + n / d \leq 2 + n / 2$ for all proper divisors $d$ of $n$.
\end{proof}

At this point, we have all the tools required to prove Theorem~\ref{thm:fk_asymp}.
\begin{proof}[Proof of Theorem~\ref{thm:fk_asymp}]
We proceed by strong induction on $k$. 
For $k = 0$, the claim follows immediately from Lemma~\ref{lem:f0_asymp}.
Suppose $k \geq 1$ and that the statement holds for all nonnegative integers $k^\prime < k$.
Then, as $n \to +\infty$, we have $f_l(n), f_r(n) \ll 4^n n^{k-3/2}$ for all nonnegative integers $l,r < k$ and applying Lemma~\ref{lem:conv_bound} we conclude that $(f_l *^\prime f_r)(n) \ll 3^n$.
In particular, the series
\begin{equation*}
\sum_{t=1}^\infty \frac{(f_l *^\prime f_r)(t)}{4^{t-1}}
\end{equation*}
converges.
Since $\T_k$ is finite and, since
\begin{equation*}
f_k(n) = \sum_{(p,\l,\r) \in \T_k} f_{(p,\l,\r)}(n) ,
\end{equation*}
it suffices to prove that for all $(p,\l,\r) \in \T_k$ we have
\begin{equation}\label{equ:fplr_asymp}
f_{(p,\l,\r)}(n) \sim \frac1{4\sqrt{\pi} p!} \, 4^n n^{p-3/2} \prod_{i=1}^p \left( \sum_{t=1}^\infty \frac{(f_{\ell_i} *^\prime f_{r_i})(t)}{4^{t-1}} \right) ,
\end{equation}
as $n \to +\infty$.
Fix $\varepsilon > 0$ and $N \in \N$.
In light of Lemma~\ref{lem:f0_asymp} and since $\binom{m}{p} \sim \tfrac{m^p}{p!}$ as $m \to +\infty$, there exists a positive integer $n_{\varepsilon,N} > N$ such that
\begin{equation*}
\binom{m}{p} f_0(m) \geq \left(\frac1{4\sqrt{\pi}p!} - \varepsilon\right) 4^m n^{p-3/2},
\end{equation*}
for all positive integers $n \geq n_{\varepsilon,N}$ and $m \in [n-N+p,n]$.
Consequently, using Lemma~\ref{lem:fplr_formula}, we obtain
\begin{align*}
f_{(p,\l,\r)}(n) &\geq \sum_{\substack{n_1 + \cdots + n_p + m = n + p \\ n_1 + \cdots + n_p \leq N}} \binom{m}{p} f_0(m) \prod_{i=1}^p (f_{\ell_i} *^\prime f_{r_i})(n_i) \\
&\geq \left(\frac1{4\sqrt{\pi}p!} - \varepsilon\right) n^{p-3/2} \sum_{\substack{n_1 + \cdots + n_p + m = n + p \\ n_1 + \cdots + n_p \leq N}} 4^m \prod_{i=1}^p (f_{\ell_i} *^\prime f_{r_i})(n_i) \\
&\geq \left(\frac1{4\sqrt{\pi}p!} - \varepsilon\right) \, 4^n n^{p-3/2} \sum_{n_1 + \cdots + n_p \leq N} \prod_{i=1}^p \frac{(f_{\ell_i} *^\prime f_{r_i})(n_i)}{4^{n_i-1}}
\end{align*}
for $n \geq n_{\varepsilon,N}$, so that
\begin{align*}
\liminf_{n \to \infty} \frac{f_{(p,\l,\r)}(n)}{4^n n^{p-3/2}} \geq \left(\frac1{4\sqrt{\pi}p!} - \varepsilon\right) \sum_{n_1 + \cdots + n_p \leq N} \prod_{i=1}^p \frac{(f_{\ell_i} *^\prime f_{r_i})(n_i)}{4^{n_i-1}} .
\end{align*}
Therefore, as $\varepsilon \to 0$ and $N \to +\infty$, we get 
\begin{align}\label{equ:fplr_liminf}
\liminf_{n \to \infty} \frac{f_{(p,\l,\r)}(n)}{4^n n^{p-3/2}} &\geq \frac1{4\sqrt{\pi}p!} \sum_{(n_1, \ldots, n_p) \in \N^p} \prod_{i=1}^p \frac{(f_{\ell_i} *^\prime f_{r_i})(n_i)}{4^{n_i-1}} \\ \nonumber
&= \frac1{4\sqrt{\pi}p!} \prod_{i=1}^p \left(\sum_{t=1}^\infty \frac{(f_{\ell_i} *^\prime f_{r_i})(t)}{4^{t-1}}\right) .
\end{align}
On the other hand, there exists $m_\varepsilon \in \N$ such that
\begin{equation*}
\binom{m}{p} f_0(m) \leq \left(\frac1{4\sqrt{\pi}p!} + \varepsilon\right) 4^m m^{p-3/2} ,
\end{equation*}
for all $m \geq m_\varepsilon$.
Thus,
\begin{align}\label{equ:big_m}
\sum_{\substack{n_1 + \cdots + n_p + m = n + p \\ m \geq m_\varepsilon}} &\binom{m}{p} f_0(m) \prod_{i=1}^p (f_{\ell_i} *^\prime f_{r_i})(n_i) \\ \nonumber &
\leq \left(\frac1{4\sqrt{\pi}p!} + \varepsilon\right) \sum_{\substack{n_1 + \cdots + n_p + m = n + p \\ m \geq m_\varepsilon}} 4^m m^{p-3/2} \prod_{i=1}^p (f_{\ell_i} *^\prime f_{r_i})(n_i) \\ \nonumber &
\leq \left(\frac1{4\sqrt{\pi}p!} + \varepsilon\right) 4^n n^{p-3/2} \sum_{n_1 + \cdots + n_p \leq n + p - m_\varepsilon} \prod_{i=1}^p \frac{(f_{\ell_i} *^\prime f_{r_i})(n_i)}{4^{n_i-1}} \\ \nonumber &
\leq \left(\frac1{4\sqrt{\pi}p!} + \varepsilon\right) 4^n n^{p-3/2} \prod_{i=1}^p \left(\sum_{t=1}^\infty \frac{(f_{\ell_i} *^\prime f_{r_i})(t)}{4^{t-1}}\right) .
\end{align}
We claim that
\begin{equation}\label{equ:big_m_little_o}
\sum_{n_1 + \cdots + n_p > n + p - m_\varepsilon} \prod_{i=1}^p \frac{(f_{\ell_i} *^\prime f_{r_i})(n_i)}{4^{n_i-1}} = o(n^{p-3/2}) ,
\end{equation}
as $n \to +\infty$.
This is straightforward if $p \geq 2$, since the left hand side of (\ref{equ:big_m_little_o}) is bounded while $n^{p-3/2} \to +\infty$. On the other hand if $p=1$ then
\begin{equation*}
\sum_{n_1 > n + 1 - m_\varepsilon} \frac{(f_{\ell_1} *^\prime f_{r_1})(n_1)}{4^{n_1-1}} = O\!\left((3/4)^n\right) = o(n^{-1/2}) ,
\end{equation*}
as $n \to +\infty$.
Hence,
\begin{align}\label{equ:little_m} 
\sum_{\substack{n_1 + \cdots + n_p + m = n + p \\ m < m_\varepsilon}} & \binom{m}{p}  f_0(m) \prod_{i=1}^p (f_{\ell_i} *^\prime f_{r_i})(n_i) \\ \nonumber &
\leq \left(\max_{m < m_\varepsilon}4^{-m}\binom{m}{p} f_0(m)\right) \sum_{\substack{n_1 + \cdots + n_p + m = n + p \\ m < m_\varepsilon}} 4^m \prod_{i=1}^p (f_{\ell_i} *^\prime f_{r_i})(n_i) \\ \nonumber &
\ll 4^n \sum_{\substack{n_1 + \cdots + n_p > n + p - m_\varepsilon}} \prod_{i=1}^p \frac{(f_{\ell_i} *^\prime f_{r_i})(n_i)}{4^{n_i-1}} 
= o(4^n n^{p-3/2})
\end{align}
as $n \to +\infty$.
Therefore, summing (\ref{equ:big_m}) and (\ref{equ:little_m}), and using Lemma~\ref{lem:fplr_formula}, we obtain
\begin{align*}
\limsup_{n \to \infty} \frac{f_{(p,\l,\r)}(n)}{4^n n^{p-3/2}} &\leq \left(\frac1{4\sqrt{\pi}p!} + \varepsilon\right) \prod_{i=1}^p \left(\sum_{t=1}^\infty \frac{(f_{\ell_i} *^\prime f_{r_i})(t)}{4^{t-1}}\right) .
\end{align*}
We conclude that as $\varepsilon \to 0$, we get
\begin{align*}
\limsup_{n \to \infty} \frac{f_{(p,\l,\r)}(n)}{4^n n^{p-3/2}} \leq \frac1{4\sqrt{\pi}p!} \prod_{i=1}^p \left(\sum_{t=1}^\infty \frac{(f_{\ell_i} *^\prime f_{r_i})(t)}{4^{t-1}}\right).
\end{align*}
Combining this with (\ref{equ:fplr_liminf}) give (\ref{equ:fplr_asymp}) concludes the proof.

\section{Proof of Theorem~\ref{thm:S_short_bound}}

Set $c := (1 - \varepsilon) / \log 4$ and for $x > 0$ define
\begin{equation*}
E(x) := \{n \leq x : S_{\mathrm{short}}(n) < c \log n\} .
\end{equation*}
For each positive integer $k$, let $\ell(k)$ be the number of exponential arithmetic formulas of length $k$.
Writing such formulas in Polish notation we see that $\ell(k) \leq 4^k$.
In fact, for each of the $k$ symbols of the Polish notation we have at most $4$ choices, corresponding to addition, multiplication, exponentiation or $1$.
Furthermore, observe that if $A_n$ denote a shortest length arithmetic formula for $n$, then clearly $A_m \neq A_n$ for all $m \neq n$.
In conclusion,
\begin{equation*}
|E(x)| \leq \sum_{k < c \log x} \ell(k) \leq \sum_{k < c \log x} 4^k = O(x^{1-\varepsilon}) ,
\end{equation*}
which is our claim.
\end{proof}

\section{Proof of Theorem~\ref{thm:S_FCF_asymp}}
Throughout this section, given a positive integer $n$, we write
\begin{equation*}
n = \sum_{j = 0}^\infty d_j(n) 2^j , \mbox{ with } d_j(n) \in \{0,1\} ,
\end{equation*}
for its binary expansion. In particular, we define
\begin{equation*}
s_2(n) := |\{j \geq 0 : d_j(n) = 1\}| = \sum_{j=0}^\infty d_j(n) ,
\end{equation*}
i.e., the number of nonzero binary digits of $n$.
Furthermore, let $\lb x := \log x / \log 2$ be the binary logarithm of $x$.

\begin{lem}\label{lem:weak}
For fixed $\varepsilon > 0 $, if
\begin{equation*}
S_\varepsilon(x) := \left\{n \leq x : s_2(n) \leq \left(\tfrac1{2} - \varepsilon\right) \!\lb x \right\},
\end{equation*}
then $|S_\varepsilon(x)| = o(x)$, as $x \to \infty$.
\end{lem}
\begin{proof}
Let $N$ be the positive integer such that $2^{N-1} \leq x < 2^N$.
Moreover, let $X_1, \ldots, X_N$ be a sequence of independent random variables with 
\begin{equation*}
\P(X_i = 0) = \P(X_i = 1) = \tfrac1{2}, \mbox{ for } i=1,\ldots,N. 
\end{equation*}
Then, for each nonnegative integer $k \leq N$,
\begin{equation*}
|\{ n < 2^N : s_2(n) = k \}| = 2^N \cdot \P(X_1 + \cdots + X_N = k) .
\end{equation*}
By the weak law of large numbers,
\begin{equation*}
\P\!\left( \Big | \frac{X_1 + \cdots + X_N}{N} - \frac{1}{2} \Big | > \varepsilon \right ) \rightarrow 0
\end{equation*}
as $N \rightarrow \infty$. 
Therefore,
\begin{align*}
|S_\varepsilon(x)| \leq \left|\left\{n < 2^N : s_2(n) \leq \left(\tfrac1{2} - \varepsilon\right) \! N \right\}\right| = o(2^N) = o(x) ,
\end{align*}
as $x \to \infty$, since $2^N \leq 2x$.
\end{proof}
We are now ready to prove Theorem~\ref{thm:S_FCF_asymp}.
\begin{proof}[Proof of Theorem~\ref{thm:S_FCF_asymp}]
Fix $\varepsilon > 0$ and let $\delta \in \;]0,\varepsilon]$ be arbitrary.
According to Lemma~\ref{lem:weak}, for $x$ sufficiently large we have $|S_\varepsilon(x)| < \delta x$ and also $|S_\varepsilon(\lb x)| < \delta \lb x$.
Let $T_\varepsilon(x) := [1,x] \setminus S_\varepsilon(x)$, so that $|T_\varepsilon(x)| >(1 - \delta)x$.
It is easily seen that $S_{\mathrm{FCF}}(n) \geq s_2(n)$ for all positive integers $n$.
Hence, for each $n \in T_\varepsilon(x)$ we have
\begin{align*}
S_{\mathrm{FCF}}(n) &\geq \sum_{\substack{j \leq \lb x \\ d_j(n) = 1}} S_{\mathrm{FCF}}(j) \geq \sum_{\substack{j \leq \lb x \\ d_j(n) = 1}} s_2(j) \geq \sum_{\substack{j \leq \lb x \\ d_j(n) = 1 \\ j \notin S_\varepsilon(\lb x)}} s_2(j) \\
&> (\tfrac1{2}-\varepsilon) \cdot \lb \lb x \sum_{\substack{j \leq \lb x \\ d_j(n) = 1 \\ j \notin S_\varepsilon(\lb x)}} 1 \\
&> (\tfrac1{2}-\varepsilon)(\tfrac1{2}-\varepsilon-\delta) \lb x \lb \lb x \\
&\geq (\tfrac1{2}-2\varepsilon)^2 \lb n \lb \lb n  .
\end{align*}
In conclusion, for any $\delta \in \;]0,\varepsilon]$ we have that for sufficiently large $x$,
\begin{equation}\label{equ:S_FCF_ineq}
S_{\mathrm{FCF}}\left(n\right) \geq \left(\frac{1}{2} - 2\varepsilon\right)^2 \cdot \lb n \lb \lb n > \left(\frac{1}{2} - 2\varepsilon\right)^2 \frac1{(\log 2)^2}\cdot \log n \log \log n,
\end{equation}
holds for at least $(1 - \delta)x$ positive integers $n \leq x$.
Therefore, (\ref{equ:S_FCF_ineq}) holds for almost all positive integers, and our claim follows.
\end{proof}

\section{Proof of Theorem~\ref{thm:S_SCF_bound}}

Fix a positive integer $n$.
In the second canonical form of $n$, we replace any occurrence of $(1 + 1)$ by the symbol $2$.
For example, after this process the second canonical form of $51$ becomes $(1 + 2)(1 + 2^{2^2})$.
Now let $t(n)$ be the number of $2$'s in this formula for $n$.
Then, upon ignoring every addition, and by repeatedly using the inequality $2^y \geq 2 \cdot y$, it follows that $n \geq 2^{t(n)}$.
To continue the example,
\begin{equation*}
51 = (1 + 2)(1 + 2^{2^2}) \geq 2 \cdot 2^{2^2} \geq 2 \cdot 2^{2 \cdot 2} \geq 2 \cdot (2\cdot (2 \cdot 2)) .
\end{equation*}
Hence $t(n) \leq \log n / \log 2$ and to prove Theorem~\ref{thm:S_SCF_bound} it is sufficient to show that $S_{\mathrm{SCF}}(n) \leq 6t(n) - 1$ for each integer $n \geq 2$.
We proceed by strong induction on $n$.
For $n = 2$ and $n = 3$ the claim is true, hence assume $n \geq 4$ and that the inequality holds for all integers in $[2,n-1]$.
If $n$ is a prime number then we have three cases:
\begin{enumerate}[(i).]
\item $n = 1 + (1 + 1) \cdot m$, with $m$ an odd integer such that $2 \leq m < n$.
\item $n = 1 + (1 + 1)^s$, with $s \geq 2$ an integer.
\item $n = 1 + (1 + 1)^s \cdot m$, with $m$ and $s$ integers such that $m$ is odd, $2 \leq m < n$ and $s \geq 2$.
\end{enumerate}
We do only case (iii), the others are similar. 
It results $t(n) = 1 + t(s) + t(m)$, so by inductive hypothesis
\begin{equation*}
S_{\mathrm{SCF}}(n) = 7 + S_{\mathrm{SCF}}(s) + S_{\mathrm{SCF}}(m) \leq 7 + (6t(s) - 1)  + (6t(m) - 1) = 6t(n) - 1 .
\end{equation*}
If $n$ is composite, let $n = p_1 \cdots p_k q_1^{b_1} \cdots q_h^{b_h}$ be its prime factorization, with $b_i \geq 2$.
We have
\begin{equation*}
t(n) = \sum_{i=1}^k t(p_i) + \sum_{j=1}^h (t(q_j) + t(b_j)) .
\end{equation*}
Since $2 \leq p_i, q_j, b_j < n$ for all $i=1,\ldots,k$ and $j=1,\ldots,h$, by inductive hypothesis we obtain
\begin{align*}
S_{\mathrm{SCF}}(n) &= k + 2h - 1 + \sum_{i=1}^k S_{\mathrm{SCF}}(p_i) + \sum_{j=1}^h (S_{\mathrm{SCF}}(q_j) + S_{\mathrm{SCF}}(b_j)) \\
&\leq 6\sum_{i=1}^k t(p_i) + 6\sum_{j=1}^h (t(q_j) + t(b_j)) - 1 \\
&= 6t(n) - 1 ,
\end{align*}
hence the proof is complete.

\section{Acknowledgements}

 We would like to thank the IAS for providing 
excellent working conditions and Noga Alon for the proof of the lower
bound for $S_{\text{short}}(n)$. 

\bibliography{counting}
\bibliographystyle{plain}

\end{document}